\documentclass[12pt, reqno]{amsart}
\usepackage{amsmath,amsfonts,amsbsy,amsgen,amscd,mathrsfs,amssymb,amsthm}
\usepackage[usenames,dvipsnames]{xcolor}
\usepackage[bookmarks=true,
bookmarksnumbered=true, breaklinks=true,linkcolor=blue, citecolor=purple,
pdfstartview=FitH, hyperfigures=false,
plainpages=false, naturalnames=true,
colorlinks=true,pagebackref=true,
pdfpagelabels]{hyperref}
\hypersetup{pdftitle={ },pdfauthor={ }}
\usepackage[a4paper,margin=1in]{geometry}
\usepackage{setspace}
\usepackage{amsmath}
\usepackage{enumerate} 
\allowdisplaybreaks[4]
\numberwithin{equation}{section}
\newtheorem{theorem}{Theorem}[section]
\newtheorem{lemma}[theorem]{Lemma}

\newtheorem{problem}[theorem]{Problem}

\title[flow by Gauss curvature to the Minkowski problem of $p$-harmonic measure]{flow by Gauss curvature to the Minkowski problem of $p$-harmonic measure}

\author{Chao Li}
\address{Chao Li:
School of Mathematics and Statistics, Ningxia University, Yinchuan, Ningxia, 750021, China.}
\address{Ningxia Basic Science Research Center of Mathematics, Ningxia University, Yinchuan 750021, China.}
	\email{lichao@nxu.edu.cn, lichao166298@163.com}
\author{Xia Zhao}
\address{Xia Zhao: School of Mathematics and Statistics, Nanjing University of Science and Technology, Nanjing, 210014, China}
\email{zhaoxia20161227@163.com}

	

\begin{document}
\begin{abstract}
The Minkowski problem of harmonic measures was first studied by Jerison \cite{JER1991}. Recently, Akman and Mukherjee \cite{AKM2023} studied the Minkowski problem corresponding to $p$-harmonic measures on convex domains and generalized Jerison's results. In this paper, we prove the existence of the smooth solution to the Minkowski problem for the $p$-harmonic measure  by method of the Gauss curvature flow.
\end{abstract}

\maketitle

\vskip 20pt
\section{Introduction and main results}

The classical Minkowski problem is one of the core problems in Brunn-Minkowski theoretical research in convex geometry. In differential geometry, the classical Minkowski problem requires the construction of a strictly compact convex hypersurface $\Omega$ with a specific Gaussian curvature. More precisely, the classical Minkowski problem is to given a finite Borel measure $\mu$ on the unit sphere $\mathbb{S}^{n-1}$, under what necessary and sufficient conditions does there exist a unique (up to translations) convex body (compact, convex subsets with non-empty interiors) $\Omega$ such that the surface area measure $S_\Omega$ of $\Omega$ is equal to $\mu$? The existence and uniqueness of solutions to the classical Minkowski problem have been solved. The solution of the Minkowski problem identified the conditions
 \begin{enumerate}[(i)]
        \item the measure $\mu$ is not concentrated on any great subsphere; that is,
$$
\int_{\mathbb{S}^{n-1}}|\langle\zeta , \xi\rangle| d \mu(\xi)>0, \quad \text { for each } \zeta \in \mathbb{S}^{n-1},
$$
\item the centroid of the measure $\mu$ is at the origin; that is,
$$
\int_{\mathbb{S}^{n-1}} \xi d \mu(\xi)=0,
$$
\end{enumerate}
on the measure as necessary and sufficient conditions for existence and uniqueness.

In recent years, the research of Minkowski problem has made a series of rich results. Since Lutwak \cite{LUT1993} introduced the concept of $L_p$-surface area measure, it has sparked a wave of research on $L_p$-Minkowski problems. With varying values of $L_p$-Minkowski problems give rise to several intriguing variant problems. When $p=1$, the $L_p$-Minkowski problem is the classical Minkowski problem; when $p=-n$, the $L_p$-Minkowski problem is the centro-affine Minkowski problem \cite{ZGX2015}; and when $p=0$, the $L_p$-Minkowski problem is the well-known log-Minkowski problem \cite{BLYZ2012,BLYZ2013,CFL2023,CLZ2019}. In recent years, the study of Minkowski problems has also led to some similar variant problems, such as Orlicz-Minkowski problem \cite{HLYZ2010,GHWXY1,GHWXY2}, dual Minkowski problem \cite{HLYZ2016,LNYZ2024,ZXY2018}, weighted Minkowski problem \cite{KLL2023,LWG2023,LIV2019},   Gaussian Minkowski problem \cite{HXZ2021},  Minkowski problem of electrostatic capacity \cite{JER1996, ZDX2020}, Minkowski problem of torsional rigidity \cite{CAM2010,HJR2024} and  Minkowski problem of chord measures\cite{LXYZ2024,XYZZ2023}.
For more research on Minkowski problem, we can refer to the reference by  B\"{o}r\"{o}czky, Ramos and Figalli \cite{BRF2024}, which introduces the continuity method and variational method of Minkowski problem solving, as well as the latest development of Minkowski problem. 

Minkowski problem has important research significance in convex geometry and the study of Minkowski problem has promoted the development of fully nonlinear partial differential equations \cite{BIJ1994}. In addition, the solution of $L_p$-Minkowski problem has proved the $L_p$ affine Sobolev inequality \cite{LYZ2002}, affine Moser-Trudinger  inequality and affine Morrey-Sobolev inequality \cite{CLYZ2009} play a key role.

In this paper, we will continue to study the Minkowski problem of a class of neglected measures, namely, the Minkowski problem of harmonic measures.  The study of Minkowski problem of harmonic measure can be traced back to Jerison \cite{JER1989} pioneering work. 
Let $\Omega$ be a convex subset of $\mathbb{R}^n$ that contains the origin 0. Harmonic measure for $\Omega$ at 0 is the measure $\omega$ on $\partial \Omega$ such that for all continuous functions $f$ on $\partial \Omega$,
$$
u(0)=\int_{\partial \Omega} f d \omega,
$$
where $u$ solves the Dirichlet problem:
\begin{align}\label{Eq:thbjz}
\begin{cases}
\operatorname{div}\left(\nabla u\right)=0 & \text { in } \Omega, \\ u=f &\text { on } \partial \Omega.\end{cases}
\end{align}
(If $\Omega$ is unbounded, then suppose further that $f$ tends to zero at infinity and $u$ is the unique solution that tends in zero at infinity.) Thus, for any convex open set $\Omega$ containing the origin we can define a measure $\mu$ on $\mathbb{S}^{n-1}$ by $\mu(\Omega)=\omega\left(\mathbf{g}_{\Omega}^{-1}\right)$ for all Borel sets $\Omega \subset \mathbb{S}^{n-1}$ (i.e., $\left.\mathbf{g}_*(d \omega)=d \mu\right)$.
Here, $\mathbf{g}_{\Omega}^{-1}$ denotes the inverse of the Gauss  map $\mathbf{g}_{\Omega}: \partial \Omega \rightarrow \mathbb{S}^{n-1}$, where $\mathbf{g}_{\Omega}(x)$ is the outer unit normal at $x \in \partial \Omega$. Harmonic measure has total mass 1, so, the problem of prescribing harmonic measure on a convex domain can be stated: 
\begin{problem}\label{pro1}
Given a probability measure $\mu$ on $\mathbb{S}^{n-1}$, find a convex open set $\Omega$ such that $\mathbf{g}_*(d \omega)=d \mu$.
\end{problem}
Meanwhile, Jerison \cite{JER1989} also studied the existence and uniqueness of the solution for sufficiently large integer $n$ for Problem \ref{pro1} by  the method of continuity. Then,  Jerison \cite{JER1991} further proves the existence, uniqueness and regularity of Problem \ref{pro1} for all non-negative integer $n$. As a generalization of harmonic measure, $p$-harmonic measure has received extensive attention.   In contrast to surface measures, harmonic measures, capacitary measures, due to the different choices of $p$-harmonic functions, the definition of $p$-harmonic measure on the boundary of the convex domain is not unique. Recently,  Akman and Mukherjee \cite{AKM2023} show that the choice of  $p$-harmonic functions can be maintained because the characteristics at the boundary are similar for all $p$-harmonic measures. Let $1<p<\infty$, for any convex set $\Omega \subset \mathbb{R}^{n}$, the $p$-harmonic measure $\omega_{p}$ supported on $\partial \Omega$ associated to a function $u=u_{\Omega}$ is given by
\begin{align*}
\omega_{p}(E)=\int_{\partial \Omega \cap E}|\nabla u|^{p-1} d \mathcal{H}^{n-1},
\end{align*}
for any measurable $E \subseteq \mathbb{R}^{n}$, where, $u$ vanish on $\partial K$ and is non-negative and $p$-harmonic in a neighborhood of it. Here, $\omega_{p}$ is the $p$-harmonic measure with respect to a function $u=u_{\Omega} \in W^{1, p}(\Omega \cap N)$ given by $d \omega_{p}=|\nabla u|^{p-1} d \mathcal{H}^{n-1}\llcorner\partial \Omega$, where $u$ is $p$-harmonic in $\Omega \cap N$, and satisfies
\begin{align}\label{Eq:pthhs}
\begin{cases}\operatorname{div}\left(|\nabla u|^{p-2} \nabla u\right)=0 & \text { in } \Omega \cap N, \\ u>0 & \text { in } \Omega, \\ u=0 & \text { on } \partial \Omega,\end{cases}
\end{align}
where $N$ is a neighbourhood of $\partial \Omega$, thus, $u \in W^{1, p}(N)$ upon zero extension. Up to possible reduction, the choice of $N$ is made so that $\nabla u \neq 0$ in $\Omega \cap N$ and
\begin{align*}
\|u\|_{L^{\infty}(\bar{N} \cap \Omega)}+\|\nabla u\|_{L^{\infty}(\bar{N} \cap \Omega)}<\infty,
\end{align*}
we also assume that $\partial N$ is $C^{\infty}$. For a general bounded convex set $K$, if $K=\bar{\Omega}$ for a convex domain $\Omega$, then $\mu_{K}=\mu_{\Omega}$; if the interior of $K$ is empty, i.e. $K=\partial K$, then $\mu_{K}=\left(\mathbf{g}_{K}\right)_{*} \omega_{p}$, where $\omega_{p}$ is the $p$-harmonic measure with respect to a function $u=u_{K} \in W^{1, p}(N)$ with $N$ being a neighbourhood of $K$, then, $d \omega_{p}=|\nabla u|^{p-1} d \mathcal{H}^{n-1}\llcorner K$, where $u$ is $p$-harmonic in $N$ and satisfies
\begin{align*}
\begin{cases}\operatorname{div}\left(|\nabla u|^{p-2} \nabla u\right)=0 & \text { in } N \backslash K, \\ u \geq 0 & \text { in } N, \\ u=0 & \text { on } K.\end{cases}
\end{align*}
Thus, in general, $\mu_{K}$ defined on $\mathbb{S}^{n-1}$ associated to $u=u_{K} \in W^{1, p}(N)$ as
\begin{align*}
\mu_{K}(E)=\int_{\mathbf{g}_{K}^{-1}(E)}|\nabla u|^{p-1} d \mathcal{H}^{n-1}, \quad \text { for any measurable } E \subseteq \mathbb{S}^{n-1}.
\end{align*}
If $K=\bar{\Omega}$ for a convex domain $\Omega$, then, (\ref{Eq:mkei}) coincides with the standard push forward measure (\ref{Eq:mKgs}).  On this basis, Akman and Mukherjee \cite{AKM2023} reaearched the following Minkowski problem of the $p$-harmonic measure: 
\begin{problem}\label{pro2}
Given a finite regular Borel measure $\mu$ on $\mathbb{S}^{n-1}$ satisfying conditions
 \begin{enumerate}[{\rm(i)}]
\setlength{\itemsep}{12pt}
\item $ \int_{\mathbb{S}^{n-1}}|\langle\zeta, \xi\rangle| d \mu(\xi)>0, \quad \forall \zeta \in \mathbb{S}^{n-1},$
\item $\int_{\mathbb{S}^{n-1}} \xi d \mu(\xi)=0,$
\item $\text {if } \mu(\{\xi\})>0 \text { then } \mu(\{-\xi\})=0,$
\end{enumerate}
is there a convex domain $\Omega$ such that
\begin{align}\label{Eq:mKgs}
\mu_{\Omega}=\mu?
\end{align}
Here, $\mu_{\Omega}=(g_{\Omega})_*\omega_p$ and $\omega_{p}$ is any $p$-harmonic measure on $\partial \Omega$ for $1<p<\infty$.
\end{problem}

Akman and Mukherjee \cite{AKM2023} proved the existence of the solution to Problem \ref{pro2}. Obviously, by (\ref{Eq:mKgs}), when a given finite Borel measure has a positive density $f(\xi)$ on the unit sphere $\mathbb{S}^{n-1}$ for a bounded convex domain $\Omega$ of class $C_{+}^2$ in $\mathbb{R}^n$, for  $1<p<\infty$,  then the existence of  Minkowski problem of $p$-harmonic measure in the smooth case amounts to solving the following Monge-Amp\`{e}re equation 
\begin{align}\label{eq:thmafc}
|\nabla u(F(\xi))|^{p-1} \operatorname{det}\left(h_{i j}+h \delta_{i j}\right)(\xi)=f(\xi), \quad \xi \in \mathbb{S}^{n-1}.
\end{align}
Motivated by \cite{CHZ2019,HJR2024}, to solve (\ref{eq:thmafc}), we instead consider the following normalized Monge-Amp\`{e}re equation, for  $1<p<\infty$, which is expressed as
\begin{align}\label{eq:thmaj}
|\nabla u(F(x))|^{p-1}\operatorname{det}\left(h_{i j}+h \delta_{i j}\right)(x)=f(x) \frac{ \Gamma(\Omega)}{\int_{\mathbb{S}^{n-1}} f(x) h(x) d x}, \quad x \in \mathbb{S}^{n-1},
\end{align}
where,  $\Gamma(\Omega)$ is defined by Akman and Mukherjee \cite{AKM2023}, see (\ref{Eq:gmjfbd}). 
By homogeneity, it is clear to see that, if $h(x)$ is a solution of (\ref{eq:thmaj}), then $\left[\frac{\Gamma(\Omega)}{\int_{\mathbb{S}^{n-1} f(x) h(x) d x}}\right]^{-\frac{1}{n+p-2}} h$ is a solution of (\ref{eq:thmafc}).  In other words, the solution of (\ref{eq:thmafc}) naturally arises as long as (\ref{eq:thmaj}) has a solution.
It is a very effective method to solve the existence of Minkowski problem by constructing flow  curvature. This method has been well developed and applied, see Reference \cite{CHZ2019,HJR2024, LYU2019}

In this article, we will focus on solving (\ref{eq:thmaj}) by method of the constructing  Gauss curvature flow. Let $\Omega_0$ be a smooth, origin symmetric and strictly convex body in $\mathbb{R}^n$ and $1<p<\infty$, we consider the following anisotropic Gauss curvature flow,
\begin{align}\label{eq:thgqll}
\left\{\begin{array}{l}
\frac{\partial X(x, t)}{\partial t}=- \eta(t)\frac{f(x) h(x, t)  \mathcal{K}(x, t)}{\mid \nabla u(F(x, t), t)\mid^{p-1}}\nu+X(x, t), \\
X(x, 0)=X_0(x),
\end{array}\right.
\end{align}
where
\begin{align}\label{eq:yita}
\eta(t)=\frac{ \Gamma\left(\Omega_t\right)}{\int_{\mathbb{S}^{n-1}} f(x) h(x, t) d x},\end{align}
where $h(x, t)$ is the support function of $\Omega_t$, $\mathcal{K}=\operatorname{det}\left(h_{i j}+h \delta_{i j}\right)^{-1}$ is the Guass curvature of strictly convex hypersurface $\partial \Omega_t$ parameterized by $X(x, t): \mathbb{S}^{n-1} \rightarrow \mathbb{R}^n$, and $\nu=x$ is the outer unit normal vector at $X(x, t)$.

By definition of support function, i.e., $h(x, t)=\langle X(x, t), \nu\rangle$, multiplying both sides of (\ref{eq:thgqll}) by $\nu$, we have
\begin{align}\label{eq:thgqwf}
\left\{\begin{array}{l}
\frac{\partial h(x, t)}{\partial t}=- \eta(t)\frac{f(x) h(x, t)  \mathcal{K}(x, t)}{\mid \nabla u(F(x, t), t)\mid^{p-1}}+h(x, t), \\
h(x, 0)=h_0(x).
\end{array}\right.
\end{align}
Moreover, we consider the corresponding functional $\Psi\left(\Omega_t\right)$ with respect to the flow (\ref{eq:thgqll}), which is defined as
\begin{align}\label{eq:gsqlfz}
\Psi\left(\Omega_t\right)=-\log \Gamma\left(\Omega_t\right)+\log \int_{\mathbb{S}^{n-1}} f(x) h(x, t) d x .
\end{align}

In this paper, by proving the long time existence of flow (\ref{eq:thgqll}), we get that $h$ is a solution to Monge-Amp\`{e}re equation (\ref{eq:thmaj}), further, we get the existence of a solution to Monge-Amp\`{e}re equation (\ref{eq:thmafc}), our results are as follows.
\begin{theorem}\label{thm11} 
Let $\Omega_0$ be a smooth, origin symmetric and strictly convex body in $\mathbb{R}^n$, $1<p<\infty$ and $f$ be an even, positive and smooth function on the unit sphere $\mathbb{S}^{n-1}$. Then, there exists a smooth, origin symmetric and strictly convex solution $\Omega_t$ satisfying (\ref{eq:thgqll}) for all time $t>0$. As a corollary, (\ref{eq:thmafc}) has a smooth, origin symmetric and strictly convex solution $\Omega$.
\end{theorem}

This paper is organized as follows. First, in the Section \ref{sec2}, we state some basic knowledge about convex geometry, and the definition and basic properties of the $p$-harmonic measure. Then, in the Section \ref{sec3}, we prove the long-time existence of the flow (\ref{eq:thgqll}). Finally,  we give the proof of Theorem \ref{thm11} in the Section \ref{sec4}.

\section{Notations and Background Materials}\label{sec2}

In this section, some notations and basic facts about convex hypersurface, convex body and $p$-harmonic measures are presented. Please refer to Reference \cite{SRA2014,UJI1991} for the related properties of   convex body and convex hypersurface, and to Reference \cite[Section 2 and 3]{AKM2023} for the related properties of $p$-harmonic measures.

\subsection{Convex body and convex hypersurface }\label{subset2.1}
Let $\Omega$ be a convex body in $\mathbb{R}^{n}$, then it's uniquely determined by its support function, $h(\Omega, \cdot):\mathbb{R}^n\rightarrow\mathbb{R}$, which is defined by 
\begin{align*}
h(\Omega, x)=\max\{ \langle x, y\rangle : y\in \Omega\}, \ \ \ \ \ \ x\in \mathbb{R}^n,
\end{align*}
where $\langle\cdot, \cdot\rangle$ denotes the standard inner product of $\mathbb{R}^{n}$.

Let $\Omega$ be a compact star-shaped set (about the origin) in $\mathbb{R}^n$, then its radial function,
$\rho_\Omega=\rho(\Omega, \cdot):\mathbb{R}^n\setminus\{0\}\rightarrow[0, \infty)$, is defined by 
\begin{align*} 
\rho(\Omega, x)=\max\{\lambda\geq0: \lambda x\in K\}, \ \ \ \ \ \ x\in\mathbb{R}^n\setminus \{0\}.
\end{align*}

 Suppose that $\Omega$ is parameterized by the inverse
Gauss map $X:\mathbb{S}^{n-1}\rightarrow \Omega$, that is $X(x)=\mathbf{g}_\Omega^{-1}(x)$. Then, the support function $h$ of $\Omega$ can be computed by
\begin{align}\label{Eq:202}h(x)={\langle x, X(x) \rangle}, \ \ \ \ \ \ x\in \mathbb{S}^{n-1},\end{align}
where $x$ is the outer normal of $\Omega$ at $X(x)$.

Let $\left\{e^{1}, \ldots, e^{n-1}\right\}$ denote the orthonormal frame field  of $\mathbb{S}^{n-1}$ so that for any $\xi \in \mathbb{S}^{n-1}$, the tangent space $T_{\xi}\left(\mathbb{S}^{n-1}\right)$ is spanned by $\left\{e^{i}(\xi)\right\}$.  Let $\nabla$ denote the gradient on the sphere $\mathbb{S}^{n-1}$.  Differentiating (\ref{Eq:202}), there has
\begin{align*}
\nabla_ih=\langle\nabla_ix,X(x)\rangle+\langle x,\nabla_iX(x)\rangle,\end{align*}
since $\nabla_iX(x)$ is tangent to $\Omega$ at $X(x)$, thus,
\begin{align*}
\nabla_ih=\langle\nabla_ix,X(x)\rangle.
\end{align*}

By differentiating (\ref{Eq:202}) twice, the second fundamental form $A_{ij}$ of $\Omega$ can be computed in terms of the support function,
\begin{align}\label{Eq:203}
A_{ij} = \nabla_{ij}h + \delta_{ij}h,\end{align}
where $\nabla_{ij}=\nabla_i\nabla_j$ denotes the second order covariant derivative with respect to $\delta_{ij}$. The induced metric matrix $g_{ij}$ of $\Omega$ can be derived by Weingarten's formula,
\begin{align}\label{Eq:204}
\delta_{ij}=\langle\nabla_ix, \nabla_jx\rangle= A_{ik}A_{lj}g^{kl}.\end{align}
The principal radii of curvature are the eigenvalues of the matrix $b_{ij} = A^{ik}g_{jk}$.
When considering a smooth local orthonormal frame on $\mathbb{S}^{n-1}$, by virtue of (\ref{Eq:203})
and (\ref{Eq:204}), there has
\begin{align}\label{Eq:205}
b_{ij} = A_{ij} = \nabla_{ij}h + h\delta_{ij}.
\end{align}
The inverse $b^{i j}$ of $b_{ij} $ represents the principal curvature of $\partial \Omega_t$. The Gauss curvature of $X(x)\in\Omega$ is given by
\begin{align}\label{Eq:206}
\mathcal{K}(x) = (\det (\nabla_{ij}h + h\delta_{ij} ))^{-1}.
\end{align}
\subsection{$p$-harmonic measure}
Given any convex set $\Omega \subset \mathbb{R}^{n}$ and  $1<p<\infty$, the $p$-harmonic measure $\omega_{p}$ supported on $\partial \Omega$ associated to a function $u=u_{\Omega}$ is given by (see \cite{AKM2023})
\begin{align*}
\omega_{p}(E)=\int_{\partial \Omega \cap E}|\nabla u|^{p-1} d \mathcal{H}^{n-1},
\end{align*}
for any measurable $E \subseteq \mathbb{R}^{n}$, where $u$ vanish on $\partial K$ and is non-negative and $p$-harmonic in a neighborhood of it. Here, for  $1<p<\infty$, $\omega_{p}$ is the $p$-harmonic measure with respect to a function $u=u_{\Omega} \in W^{1, p}(\Omega \cap N)$ given by $d \omega_{p}=|\nabla u|^{p-1} d \mathcal{H}^{n-1}\llcorner\partial \Omega$, where $u$ is $p$-harmonic in $\Omega \cap N$ and satisfies
\begin{align*} 
\begin{cases}
\operatorname{div}\left(|\nabla u|^{p-2} \nabla u\right)=0 & \text { in } \Omega \cap N, \\ 
u>0 & \text { in } \Omega,\\ 
u=0 & \text { on } \partial \Omega,
\end{cases}
\end{align*}
where $N$ is a neighbourhood of $\partial \Omega$; thus, $u \in W^{1, p}(N)$ upon zero extension. Up to possible reduction, the choice of $N$ is made so that $\nabla u \neq 0$ in $\Omega \cap N$ and
\begin{align*}
\|u\|_{L^{\infty}(\bar{N} \cap \Omega)}+\|\nabla u\|_{L^{\infty}(\bar{N} \cap \Omega)}<\infty
\end{align*}
and we also assume that $\partial N$ is $C^{\infty}$. For a general bounded convex set $K$, if $K=\bar{\Omega}$ for a convex domain $\Omega$, then, $\mu_{K}=\mu_{\Omega}$; if $K$ is of empty interior, i.e. $K=\partial K$, then, $\mu_{K}=\left(\mathbf{g}_{K}\right)_{*} \omega_{p}$, where $\omega_{p}$ is the $p$-harmonic measure with respect to a function $u=u_{K} \in W^{1, p}(N)$ with $N$ being a neighbourhood of $K$, given by $d \omega_{p}=|\nabla u|^{p-1} d \mathcal{H}^{n-1}\llcorner K$ where $u$ is $p$-harmonic in $N$ and satisfies
\begin{align*}
\begin{cases}
\operatorname{div}\left(|\nabla u|^{p-2} \nabla u\right)=0 & \text { in } N \backslash K, \\
 u \geq 0 & \text { in } N, \\ u=0 & \text { on } K.
\end{cases}
\end{align*}
Thus, in general, $\mu_{K}$ defined on $\mathbb{S}^{n-1}$ associated to $u=u_{K} \in W^{1, p}(N)$ as
\begin{align}\label{Eq:mkei}
\mu_{K}(E)=\int_{\mathbf{g}_{K}^{-1}(E)}|\nabla u|^{p-1} d \mathcal{H}^{n-1}, \quad \text { for any measurable } E \subseteq \mathbb{S}^{n-1}.
\end{align}

Let $\Omega$ be strongly convex domain of class $C_{+}^{2, \alpha}$ so that $\mathbf{g}_{\Omega}: \partial \Omega \rightarrow \mathbb{S}^{n-1}$ is a diffeomorphism. Let its support function be $h_{\Omega}$.  Akman and Mukherjee \cite{AKM2023} obtained
\begin{align}\label{Eq:mcdjs}
d \mu_{\Omega}=\left|\nabla u\left(F_{\Omega}(\xi)\right)\right|^{p-1} d \mathcal{H}^{n-1}\left\llcorner\partial \Omega=\left|\nabla u\left(F_{\Omega}(\xi)\right)\right|^{p-1} \operatorname{det}(\nabla_{ij}h(\xi) + h(\xi)\delta_{ij} ) \right.,
\end{align}
where $F_{\Omega}(\xi):=\mathbf{g}_{\Omega}{ }^{-1}(\xi)=\nabla h_{\Omega}(\xi)$.
Henceforth, we shall denote $h=h_{\Omega}$ and $F=F_{\Omega}$. Thus,
\begin{align}
h(\xi)=h_{\Omega}(\xi)=\left\langle\xi, \mathbf{g}_{\Omega}{ }^{-1}(\xi)\right\rangle=\left\langle x, \mathbf{g}_{\Omega}(x)\right\rangle,
\end{align}
 and
\begin{align}\label{Eq:fgh}
F(\xi)=F_{\Omega}(\xi)=\mathbf{g}_{\Omega}^{-1}(\xi)=\nabla h(\xi).
\end{align}
Sometimes, without causing confusion, we overuse the notation, we use $h_{i j}=\nabla_{ij}h$ denotes the second order covariant derivatives of $h(\Omega, \cdot)$ on $\mathbb{S}^{n-1}$ with respect to a local orthonormal frame. Then, (see, \cite[Page 258]{HJR2024})
\begin{align}\label{Eq:fgh1}
\nabla h_{\Omega}(\xi)=\sum_i h_i e^i+h \xi, \quad F_i(\xi)=\sum_j\operatorname{det}(\nabla_{ij}h(\xi) + h(\xi)\delta_{ij} ) e^j.
\end{align}
Since any $\xi\in\mathbb{S}^{n-1}$ is the outerunit normal at $F(\xi)\in\partial \Omega$ as $F=\mathbf{g}_{\Omega}^{-1}(\xi)$, thus, we have
\begin{align}\label{Eq:tdxld}
\xi=-\nabla u(F(\xi))/|\nabla u(F(\xi))|.
\end{align}

For any integrable function $f: \partial \Omega \rightarrow \mathbb{R}$, Akman and Mukherjee \cite{AKM2023} obtained
\begin{align*}
\int_{\partial \Omega} f(x) d \omega_{p}(x)=\int_{\mathbb{S}^{n-1}} f\left(\mathbf{g}_{\Omega}^{-1}(\xi)\right)\left|\nabla u\left(\mathbf{g}_{\Omega}^{-1}(\xi)\right)\right|^{p-1} \operatorname{det}(\nabla_{ij}h(\xi) + h\delta_{ij}(\xi) ) d \xi,
\end{align*}
and defined the following functional $\Gamma(K)$ for the harmonic measure of a convex set $K$
\begin{align}\label{Eq:gmjfbd}
\Gamma(K)=\int_{\mathbb{S}^{n-1}} h_{K}(\xi) d \mu_{K}(\xi),
\end{align}
and if  $\bar{\Omega}=K$, then, $\Gamma(\Omega)=\Gamma(K)$, and $\Gamma\left(K_{j}\right) \rightarrow \Gamma(K)$ uniformly if $d_{\mathcal{H}}\left(K_{j}, K\right) \rightarrow 0^{+}$as $j \rightarrow \infty$.

Hence, by (\ref{Eq:mcdjs}) and (\ref{Eq:gmjfbd}), we have
\begin{align*}
\Gamma(K)=\int_{\mathbb{S}^{n-1}} h_{K}(\xi)|\nabla u(F(\xi))|^{p-1} \operatorname{det}(\nabla_{ij}h(\xi) + h(\xi)\delta_{ij} ) d(\xi).
\end{align*}

\section{Long-time existence of the flow}\label{sec3}
In this section, we aim to get the long-time existence of the solution of the flow (\ref{eq:thgqll}), which means that we get that the flow (\ref{eq:thgqll}) is uniformly parabolic. The key is to derive the upper and lower bounds of the principal curvature. Therefore, we need to build a $C^0$, $C^1$ estimate for the flow (\ref{eq:thgqll}). First, we show that $\Gamma(\Omega_t)$ is unchanged along the flow (\ref{eq:thgqll}).

\begin{lemma}\label{lemlxx} 
For  $1<p<\infty$,  let $\Omega_t$ be a smooth, origin-symmetric and strictly convex solution satisfying the  flow (\ref{eq:thgqll}) in $\mathbb{R}^n$. Then, $\Gamma(\Omega_t)$ is unchanged  along the flow (\ref{eq:thgqll}), i.e.
\begin{align*}
\Gamma(\Omega_t)=\Gamma(\Omega_0).
\end{align*}
\end{lemma}

\begin{proof}
Let $h(\cdot,t)$   be the support function of $\Omega_t$, we have
\begin{align*}
\frac{d}{dt}\Gamma(\Omega_t)=& \int_{\mathbb{S}^{n-1}}  \frac{\partial  h}{\partial t} (x,t)  d \mu\left(\Omega_0, u\right)\\
=& \int_{\mathbb{S}^{n-1}}  \frac{\partial  h}{\partial t} (x,t) |\nabla u|^{1-p} \det(\nabla_{ij}h + h\delta_{ij})dx\\
=& \int_{\mathbb{S}^{n-1}} \bigg(-\eta(t)|\nabla u|^{1-p} \mathcal{K}(x, t)h(x, t) f(x)+h(x, t)\bigg)|\nabla u|^{p-1} \mathcal{K}^{-1}(x)dx\\
=&0.
\end{align*}
This ends the proof of Lemma \ref{lemlxx}.
\end{proof}

The following lemma shows that the functional  $\Psi(\Omega_t)$ is non-increasing along the flow (\ref{eq:thgqll}).
 
\begin{lemma}\label{lemddx}
For  $1<p<\infty$,   let $\Omega_t$ be a smooth, origin-symmetric and strictly convex solution satisfying the  flow (\ref{eq:thgqll}) in $\mathbb{R}^n$. Then, the functional (\ref{eq:gsqlfz}) is non-increasing along the flow (\ref{eq:thgqll}). Namely, 
$$\frac{d}{dt}\Psi(\Omega_t)\leq0,$$
 and the equality holds if and only if $\Omega_t$ satisfy (\ref{eq:thmaj}).
\end{lemma}
\begin{proof}
By (\ref{eq:gsqlfz}), (\ref{eq:thgqwf}), (\ref{eq:yita}) and H\"{o}lder inequality, we have
\begin{align*}
& \frac{d}{dt}\Psi(\Omega_t)\\
=&-\frac{1}{\Gamma(\Omega_t)}\int_{\mathbb{S}^{n-1}}|\nabla u|^{p-1}\mathcal{K}^{-1} \frac{\partial  h}{\partial t}dx+\frac{\int_{\mathbb{S}^{n-1}} f(x) \frac{\partial  h}{\partial t} d x}{\int_{\mathbb{S}^{n-1}} f(x) h(x, t) d x} \\
=&\int_{\mathbb{S}^{n-1}}\frac{-h \frac{\partial  h}{\partial t}}{\Gamma(\Omega_t)|\nabla u|^{-(p-1)}\mathcal{K}h}dx+\frac{\int_{\mathbb{S}^{n-1}} f(x) \frac{\partial  h}{\partial t} d x}{\int_{\mathbb{S}^{n-1}} f(x) h(x, t) d x} \\
=&\int_{\mathbb{S}^{n-1}}\partial_th\bigg(\frac{-h}{\Gamma(\Omega_t)|\nabla u|^{-(p-1)}\mathcal{K}h}+\frac{f(x)}{\int_{\mathbb{S}^{n-1}} f(x) h(x, t) d x}\bigg)dx \\
=&\int_{\mathbb{S}^{n-1}}\partial_th\bigg(\frac{-h+\frac{f(x)\Gamma(\Omega_t)|\nabla u|^{-(p-1)}\mathcal{K}h}{\int_{\mathbb{S}^{n-1}} f(x) h(x, t) d x}}{\Gamma(\Omega_t)|\nabla u|^{-(p-1)}\mathcal{K}h}\bigg)dx \\
=&\int_{\mathbb{S}^{n-1}}\partial_th\bigg(\frac{-h+\frac{f(x)\Gamma(\Omega_t)|\nabla u|^{-(p-1)}\mathcal{K}h}{\int_{\mathbb{S}^{n-1}} f(x) h(x, t) d x}}{\Gamma(\Omega_t)|\nabla u|^{-(p-1)}\mathcal{K}h}\bigg)dx \\
=&-\int_{\mathbb{S}^{n-1}}\frac{\bigg(-\frac{f(x)\Gamma(\Omega_t)|\nabla u|^{-(p-1)}\mathcal{K}h}{\int_{\mathbb{S}^{n-1}} f(x) h(x, t) d x}+h\bigg)^2}{\Gamma(\Omega_t)|\nabla u|^{-(p-1)}\mathcal{K}h}dx \\
 \leq& 0,
\end{align*}
with equality if and only if $\partial_th=0$, i.e.,
$$
\eta(t)\frac{f(x) h(x, t)  \mathcal{K}(x, t)}{\mid \nabla u(F(x, t), t)\mid^{p-1}}=h(x, t),
$$
which illustrates that $\Omega_t$ satisfies (\ref{eq:thmaj}). The proof is completed.
\end{proof}

Next, we aim to establish the $C^0, C^1$ estimates for the solution to  flow (\ref{eq:thgqll}).  To do this, we need the following lemma.

\begin{lemma}\cite[Lemma 4.28]{LNM2010}\label{lemtdyj}
 Let $\Omega \subset \mathbb{R}^n$ be a bounded Lipschitz domain. Given $1<p<\infty$, $w \in \partial \Omega$, $0<r<r_0$, suppose that $u$ is a positive $p$-harmonic function in $\Omega \cap B(w, 2 r)$. Assume also that $u$ is continuous in $\bar{\Omega} \cap \bar{B}(w, 2 r)$ and $u=0$ on $\Delta(w, 2 r)$. Then, there exist $\xi \in \partial B(0,1)$ and $c_3, \delta_{+}>1$, both of which only depend on $p, n$ and $M$ such that
$$
\delta_{+}^{-1} \frac{u(x)}{d(x, \partial \Omega)} \leq\langle\nabla u(x), \xi\rangle \leq|\nabla u(x)| \leq \delta_{+} \frac{u(x)}{d(x, \partial \Omega)},
$$
whenever $x \in \Omega \cap B\left(w, r / c_3\right)$. Moreover, $\xi$ can be chosen independently of $u$.
\end{lemma}

\begin{lemma}\label{lemco} 
Let $f$ be an even, smooth and positive function on ${\mathbb{S}^{n-1}}$, $1<p<\infty$ and $\Omega_{t}$ be a smooth, origin-symmetric and strictly convex solution satisfying the  flow (\ref{eq:thgqll}). Then
\begin{align}\label{eq:zcC0}
\frac{1}{C}\leq h(x,t)\leq C, \ \forall (x,t)\in {{\mathbb{S}^{n-1}}}\times (0,+\infty),
\end{align}
and
\begin{align}\label{eq:jxC0}
\frac{1}{C}\leq \rho(u,t)\leq C, \ \forall (u,t)\in {{\mathbb{S}^{n-1}}}\times (0,+\infty).
\end{align}
Here $h(x,t)$ and $\rho(u,t)$ are the support function and the radial function of $\Omega_{t}$, respectively.
\end{lemma}

\begin{proof}
Due to $\rho(u,t)u=\nabla h(x,t)+h(x,t)x$. Clearly, one see
\begin{align*}
\min_{{\mathbb{S}^{n-1}}} h(x,t)\leq \rho(u,t)\leq \max_{{\mathbb{S}^{n-1}}}h(x,t).
\end{align*}
This implies that the estimate (\ref{eq:zcC0}) is tantamount to the estimate (\ref{eq:jxC0}). So, we only need to establish (\ref{eq:zcC0}) or (\ref{eq:jxC0}).

Using the monotonicity of $\Psi(\Omega_{t})$, we have
\begin{align}\label{khhd}
\Psi\left(\Omega_0\right)+\log \Gamma\left(\Omega_0\right) \geq \log \int_{\mathbb{S}^{n-1}} f(x) h(x, t) d x.
\end{align}
Let $\rho_{\max}(t)=\max_{{\mathbb{S}^{n-1}}}\rho(\cdot,t)$. By a rotation of coordinate, we suppose that $\rho_{\max}(t)=\rho(e_{1},t)$. Since $\Omega_{t}$ is origin symmetric, by the definition of $h(x,t)$, one see $h(x,t)\geq \rho_{\max}(t)\langle x, e_1\rangle$ for $\forall x\in {\mathbb{S}^{n-1}}$.

Next we will obtain the upper bound of $h(x,t)$. By 
 (\ref{khhd}), 
\begin{align*}
\Gamma\left(\Omega_0\right) e^{\Psi\left(\Omega_0\right)}\geq \int_{\mathbb{S}^{n-1}} h(x, t) f d x & \geq \int_{\mathbb{S}^{n-1}} \rho_{\max }\langle x, e_1\rangle f d x \\
& \geq|\min f| \rho_{\max } \int_{\mathbb{S}^{n-1}}\langle x, e_1\rangle d x \\
& \geq|\min f| \rho_{\max } c_0,
\end{align*}
hence,
\begin{align}\label{eqyzsj}
\rho_{\max }(t) \leq \frac{\Gamma\left(\Omega_0\right) e^{\Psi\left(\Omega_0\right)}}{c_0|\min f|} \leq C,
\end{align}
for some $C>0$, independent of $t$. This implies the upper bound for (\ref{eq:zcC0}).

On the other hand, to give the lower bound of $h(x, t)$ along the following line. Note that $\operatorname{V}\left(\Omega_t\right) \leq w_{-}\left(\Omega_t\right) w_{+}\left(\Omega_t\right)^{n-1}$, where $w_{-}\left(\Omega_t\right)$ and $w_{+}\left(\Omega_t\right)$ are respectively the minimum width and maximum width of $\Omega_t$. Since $\Omega_t$ is origin-symmetric, $w_{-}\left(\Omega_t\right)=2 \min _{x \in \mathbb{S}^{n-1}} h(x, t)$ and $w_{+}\left(\Omega_t\right)=2 \max _{x \in \mathbb{S}^{n-1}} h(x, t)$. So, utilizing (\ref{eqyzsj}), we have
\begin{align}\label{eq:VMIN}
\nonumber\operatorname{V}\left(\Omega_t\right) & \leq w_{-}\left(\Omega_t\right) w_{+}\left(\Omega_t\right)^{n-1} \\
& \leq C \min _{x \in \mathbb{S}^{n-1}} h(x, t),
\end{align}
for a positive constant $C$, independent of $t$. Lemma \ref{lemtdyj}, Lemma \ref{lemlxx}, and (\ref{eqyzsj}) show that $\operatorname{V}\left(\Omega_t\right) \geq \frac{1}{C^{p-1}} \Gamma\left(\Omega_t\right)>0$ for a positive constant $C$ as showed in (\ref{eqyzsj}), together this with (\ref{eq:VMIN}), we conclude that $h(x, t)$ has a uniformly lower bound. 
\end{proof}

The $C^{1}$ estimate naturally follows by applying above $C^{0}$ estimate.
\begin{lemma}\label{lemcqgj}
 Let $f$ be an even, smooth and positive function on ${\mathbb{S}^{n-1}}$,  $1<p<\infty$ and $\Omega_{t}$ be a smooth, origin-symmetric and strictly convex solution satisfying the  flow (\ref{eq:thgqll}). Then
\begin{align}\label{Eq:c1zc}
|\nabla h(x,t)|\leq C, \ \forall (x,t)\in {{\mathbb{S}^{n-1}}}\times (0,+\infty),
\end{align}
and
\begin{align}\label{Eq:cejx}
|\nabla\rho(u,t)|\leq C, \ \forall (u,t)\in {{\mathbb{S}^{n-1}}}\times (0,+\infty),
\end{align}
for some $C> 0$, independent of $t$.
\end{lemma}
\begin{proof}
Let $u$ and $x$ be related by $\rho(u,t)u=\nabla h(x,t)+h(x,t)x$, we have
\begin{align*}
 h=\frac{\rho^{2}}{\sqrt{|\nabla \rho|^{2}+\rho^{2}}},\quad \rho^{2}=h^{2}+|\nabla h|^{2}.
\end{align*}
The above facts together with Lemma \ref{lemco} illustrate the desired result.
\end{proof}

In order to get the $C^2$ on the solution of the  flow (\ref{eq:thgqll}), we will need the following conclusions about harmonic  functions.

\begin{lemma}\cite[Lemma 3.44]{AKM2023}\label{lemejtg}
 Let $u: \Omega \rightarrow \mathbb{R}$ be as in (\ref{Eq:pthhs}) and $\left\{e^{1}, \ldots, e^{n-1},\right\}$ be an orthonormal frame field of $\mathbb{S}^{n-1}$ such that for any $\xi \in \mathbb{S}^{n-1}$ the unit vectors $e^{i}=e^{i}(\xi) \in \mathbb{S}^{n-1}$ span the tangent space $T_{\xi}\left(\mathbb{S}^{n-1}\right)$. Then
\begin{itemize}
\setlength{\itemsep}{2pt}
\item[(a)] 
 $\left\langle D^{2} u(F(\xi)) e^{i}, e^{j}\right\rangle=-\mathcal{K}(F(\xi))|\nabla u(F(\xi))| \mathcal{C}_{i, j}\left[\nabla^{2} h+h \delta_{ij}\right]$;

\item[(b)]  $\left\langle D^{2} u(F(\xi)) \xi, e^{i}\right\rangle=-\mathcal{K}(F(\xi)) \sum_{j} \mathcal{C}_{i, j}\left[\nabla^{2} h+h\delta_{ij}\right] \nabla_{j}(|\nabla u(F(\xi))|)$;

\item[(c)]  $\left\langle D^{2} u(F(\xi)) \xi, \xi\right\rangle=\frac{1}{(p-1)} \mathcal{K}(F(\xi))|\nabla u(F(\xi))| \operatorname{Tr}\left(\mathcal{C}\left[\nabla^{2} h+h \delta_{ij}\right]\right)$,
\end{itemize}
where $\mathcal{C}_{i, j}[\cdot]=\left\langle\mathcal{C}[\cdot] e^{j}, e^{i}\right\rangle$ are entries of the cofactor matrix for $i, j \in\{1, \ldots, n-1\}$ with respect to this frame and $F(\xi)=\mathbf{g}_{\Omega}{ }^{-1}(\xi)=\nabla h(\xi)$.
\end{lemma}

\begin{lemma}\cite[Proposition 3.20]{AKM2023}\label{lemtdgji} 
If $u(\cdot, t) \in W^{1, p}\left(\Omega^{t} \cap N\right)$ is the solution of (\ref{Eq:pthhs}), then, the following holds:
\begin{itemize}
\setlength{\itemsep}{2pt}
\item[(a)]  $t \mapsto u(\cdot, t)$ is differentiable at $t=0$ for all $x \in \bar{\Omega} \cap N$ and $\dot{u} \in C^{2, \beta}(\overline{\Omega \cap N})$;

\item[(b)] $\dot{u}(x)=-\langle\nabla u(x), x\rangle$ for all $x \in \partial N \cap \Omega$;

\item[(c)] $\dot{u}(x)=|\nabla u(x)| v\left(\mathbf{g}_{\Omega}(x)\right)$ for all $x \in \partial \Omega$.
\end{itemize}
\end{lemma}

In what follows, we shall give the upper and lower bounds of principal curvature of $\partial \Omega_t$ based on above preparations.
\begin{lemma}\label{thmcegj}
 Let $f$ be an even, smooth and positive function on $\mathbb{S}^{n-1}$,  $1<p<\infty$ and $\Omega_t$ be a smooth, origin-symmetric and strictly convex solution satisfying the  flow (\ref{eq:thgqll}). Then,
$$
\frac{1}{C} \leq \mathcal{K}_i \leq C, \forall(x, t) \in \mathbb{S}^{n-1} \times(0,+\infty)
$$
for some positive constants $C$, independent of $t$.
\end{lemma}
\begin{proof}
For  $1<p<\infty$, in order to prove that the Gaussian curvature $\mathcal{K}$ has an upper bound, we construct the following auxiliary function: 
\begin{align}\label{eq:fzhsq}
Q(x, t)=\frac{\eta(t)\frac{f(x) h(x, t)  \mathcal{K}(x, t)}{\mid \nabla u(F(x, t), t)\mid^{p-1}}-h(x, t)}{h-\varepsilon_0}=\frac{-h_t}{h-\varepsilon_0},
\end{align}
where $h_t=\frac{\partial h(x, t)}{\partial t}$ and
\begin{align*}
\varepsilon_0=\frac{1}{2} \min _{\mathbb{S}^{n-1} \times(0,+\infty)} h(x, t)>0 .
\end{align*}

For any fixed $t \in(0, \infty)$, assume that $Q(x, t)$ reaches its maximum value at $x_0$.  Thus, we obtain that at $x_0$,
\begin{align}\label{eq:qxyjd}
0=\nabla_i Q=\frac{-h_{t i}}{h-\varepsilon_0}+\frac{h_t h_i}{\left(h-\varepsilon_0\right)^2} .
\end{align}
Then, with the assistance of (\ref{eq:qxyjd}) and (\ref{eq:thgqwf}), at $x_0$, we deduce
\begin{align*}
0 \geq \nabla_{i i} Q & =\frac{-h_{t i i}}{h-\varepsilon_0}+\frac{2 h_{t i} h_i+h_t h_{i i}}{\left(h-\varepsilon_0\right)^2}-\frac{2 h_t h_i^2}{\left(h-\varepsilon_0\right)^3} \\
& =\frac{-h_{t i i}}{h-\varepsilon_0}+\frac{h_t h_{i i}}{\left(h-\varepsilon_0\right)^2}.
\end{align*}
From the above formula combined with (\ref{eq:fzhsq}), we get
\begin{align}\label{Eq:tzzgj}
\nonumber-h_{t i i}-h_t\delta_{ij} & \leq-\frac{h_t h_{i i}}{h-\varepsilon_0}-h_t\delta_{ij} \\
& =\frac{-h_t}{h-\varepsilon_0}\left[h_{i i}+\left(h-\varepsilon_0\right)\delta_{ij}\right] \\
\nonumber& =Q\left(b_{i i}-\varepsilon_0\delta_{ij}\right) .
\end{align}
Now take the derivative of the second variable of $Q(x,t)$, by (\ref{eq:fzhsq}) and (\ref{eq:thgqwf}), we get
\begin{align}\label{Eq:qyjgj}
\nonumber\frac{\partial Q}{\partial t}=& \frac{-h_{t t}}{h-\varepsilon_0} +\frac{h_t^2}{\left(h-\varepsilon_0\right)^2} \\
= & \frac{f(x)}{h-\varepsilon_0}\left[\frac{\partial\eta(t)}{\partial t}\frac{h(x, t)  \mathcal{K}(x, t)}{\mid \nabla u(F(x, t), t)\mid^{p-1}}+\frac{\partial h(x, t)}{\partial t}\frac{\eta(t)\mathcal{K}(x, t)}{\mid \nabla u(F(x, t), t)\mid^{p-1}}\right.\\
\nonumber&\left. +\frac{\partial \mathcal{K}(x, t)}{\partial t}\frac{\eta(t)h(x, t)}{\mid \nabla u(F(x, t), t)\mid^{p-1}}+\eta(t)h(x, t)\mathcal{K}(x, t)\frac{\partial \mid \nabla u(F(x, t), t)\mid^{1-p}}{\partial t}\right]\\
\nonumber& +Q+Q^2.
\end{align}
Next, our division makes an estimate of (\ref{Eq:qyjgj}).  From (\ref{eq:yita}), Lemma \ref{lemlxx} and (\ref{eq:fzhsq}), we obtain
\begin{align}\label{Eq:fzhsgj}
\nonumber\frac{\partial\eta(t)}{\partial t}=&\frac{\partial}{\partial t}\bigg(\frac{ \Gamma\left(\Omega_t\right)}{\int_{\mathbb{S}^{n-1}} f(x) h(x, t) d x}\bigg)\\
\nonumber=&\frac{\partial  \Gamma\left(\Omega_t\right)}{\partial  t}\bigg(\frac{1}{\int_{\mathbb{S}^{n-1}} f(x) h(x, t) d x}\bigg)+\Gamma\left(\Omega_t\right)\frac{\partial \bigg(\frac{1}{\int_{\mathbb{S}^{n-1}} f(x) h(x, t) d x}\bigg)}{\partial  t}\\
\nonumber=&-\Gamma\left(\Omega_t\right)\frac{\int_{\mathbb{S}^{n-1}} f(x) \frac{\partial h}{\partial t} d x}{\bigg(\int_{\mathbb{S}^{n-1}} f(x) h(x, t) d x\bigg)^2}\\
=&\Gamma\left(\Omega_t\right)\frac{\int_{\mathbb{S}^{n-1}} f(x) (h(x, t)-\varepsilon_0)Q(x,t) d x}{\bigg(\int_{\mathbb{S}^{n-1}} f(x) h(x, t) d x\bigg)^2}\\
\nonumber\leq&Q(x_0,t)\Gamma\left(\Omega_t\right)\frac{1}{\bigg(\int_{\mathbb{S}^{n-1}} f(x) h(x, t) d x\bigg)}\\
\nonumber\leq&C_1Q(x_0,t).
\end{align}
 Rotate the axes so that $\left\{b_{i j}\right\}$ is diagonal at $x_0$ with $b_{i j}=h_{i j}+h \delta_{i j}$, $\left\{b^{i j}\right\}$  is the inverse of $\left\{b_{i j}\right\}$ . Using (\ref{Eq:206})  and (\ref{Eq:tzzgj}) at $x_0$, we can get
\begin{align}\label{Eq:gsqlgj}
\nonumber\frac{\partial \mathcal{K}(x, t)}{\partial t}&=\frac{\partial\left[\operatorname{det}\left(\nabla^2 h+h I\right)\right]^{-1}}{\partial t}\\ 
\nonumber& =-\left[\operatorname{det}\left(\nabla^2 h+h I\right)\right]^{-2} \sum_i \frac{\partial\left[\operatorname{det}\left(\nabla^2 h+h I\right)\right]}{\partial b_{i i}}\left(h_{t i i}+h_t\right) \\
\nonumber& \leq\left[\operatorname{det}\left(\nabla^2 h+h I\right)\right]^{-2} \sum_i \frac{\partial\left[\operatorname{det}\left(\nabla^2 h+h I\right)\right]}{\partial b_{i i}} Q\left(b_{i i}-\varepsilon_0\right) \\
& \leq Q\left[\operatorname{det}\left(\nabla^2 h+h I\right)\right]^{-2}\left[\operatorname{det}\left(\nabla^2 h+h I\right)\right] \sum_i  b^{i i}\left(b_{i i}-\varepsilon_0\right) \\
\nonumber& =\mathcal{K} Q\left[(n-1)-\varepsilon_0 \sum_i b^{i i}\right]\\ \nonumber&=\mathcal{K} Q\left[(n-1)-\varepsilon_0 H\right] \\
\nonumber& \leq \mathcal{K} Q\left[(n-1)-\varepsilon_0(n-1) \mathcal{K}^{\frac{1}{n-1}}\right],
\end{align}
where $H$ denotes the mean curvature of $\partial \Omega_t$, and the last inequality stems from $H \geq(n-$ 1) $\left(\Pi_i b^{i i}\right)^{\frac{1}{n-1}}=(n-1) 
\mathcal{K}^{\frac{1}{n-1}}$.

In addition, from (\ref{Eq:fgh}), (\ref{Eq:fgh1}), (\ref{Eq:tdxld}) and (c) of Lemma \ref{lemtdgji}, we obtain
\begin{align}\label{Eq:tdwfgj}
\nonumber& \frac{\partial\mid \nabla u(F(x, t), t)\mid^{1-p}}{\partial t} \\
\nonumber= & (1-p)|\nabla u(F(x, t), t)|^{-p} \frac{\partial|\nabla u(F(x, t), t)|}{\partial t} \\
\nonumber= &-(1-p)|\nabla u(F(x, t), t)|^{-p} \left(\bigg\langle \nabla^2 u(F(x, t), t) x , \sum_i\left(h_{t i} e^i+h_t x\right)\bigg\rangle +\bigg\langle \nabla \dot{u}(F(x, t), t) , x\bigg\rangle \right) \\
\nonumber= & (1-p)|\nabla u(F(x, t), t)|^{-p}\bigg\langle\nabla^2 u(F(x, t), t) x ,\left(\sum_i\left(\left(h-\varepsilon_0\right) Q\right)_i e^i+\left(h-\varepsilon_0\right) Q x\right)\bigg\rangle \\
\nonumber&+(p-1)|\nabla u(F(x, t), t)|^{-p} \langle\nabla \dot{u}(F(x, t), t), x\rangle \\
\nonumber= & (1-p)|\nabla u(F(x, t), t)|^{-p}\left(\bigg\langle\nabla^2 u(F(x, t), t) x , \sum_i h_i Q e^i\bigg\rangle+\langle\nabla^2 u(F(x, t), t) x ,\left(h-\varepsilon_0\right) Q x\rangle\right) \\
&+(p-1)|\nabla u(F(x, t), t)|^{-p} \langle\nabla \dot{u}(F(x, t), t) , x\rangle .
\end{align}
Now, our division estimates (\ref{Eq:tdwfgj}). Since  $1<p<\infty$, from Lemma \ref{lemco} and Lemma \ref{lemcqgj}, there exist positive constants $c_1$ and $c_2$, independent of $t$, such that 
\begin{align*}
&(1-p)|\nabla u(F(x, t), t)|^{-p}\left(\bigg\langle\nabla^2 u(F(x, t), t) x , \sum_i h_i Q e^i\bigg\rangle+\langle\nabla^2 u(F(x, t), t) x ,\left(h-\varepsilon_0\right) Q x\rangle\right) \\
=&(1-p)|\nabla u(F(x, t), t)|^{-p}\left\langle\nabla^2 u(F(x, t), t) x , \sum_i h_i Q e^i\right\rangle\\&+(1-p)|\nabla u(F(x, t), t)|^{-p}\left\langle\nabla^2 u(F(x, t), t) x ,\left(h-\varepsilon_0\right) Q x\right\rangle \\
 \leq &c_1 Q|\nabla u(F(x, t), t)|^{-p}\left|\nabla^2 u(F(x, t), t)\right| + c_2 Q|\nabla u(F(x, t), t)|^{-p}\left|\nabla^2 u(F(x, t), t)\right|,
\end{align*}
meanwhile, by (c) of Lemma \ref{lemtdgji}, (\ref{Eq:fgh1}),  (\ref{Eq:tdxld}) and (\ref{eq:fzhsq}), there exist positive constants $c_3$ and $c_4$, independent of $t$, such that
\begin{align*}
& (p-1)|\nabla u(F(x, t), t)|^{-p} \langle\nabla \dot{u}(F(x, t), t) , x\rangle \\
= & (p-1)|\nabla u(F(x, t), t)|^{-p}\left\langle\nabla\left\langle-\nabla u(F(x, t), t),\left(\sum_i h_{t i} e^i+h_t x\right)\right\rangle, x\right\rangle \\
= & (p-1)|\nabla u(F(x, t), t)|^{-p}\left\langle\nabla\left\langle|\nabla u(F(x, t), t)| x ,\left(\sum_i h_{t i} e^i+h_t x\right)\right\rangle , x\right\rangle \\
= & (p-1)|\nabla u(F(x, t), t)|^{-p}\left\langle\nabla\left\langle|\nabla u(F(x, t), t)| x ,\left(\sum_i-h_i Q e^i-\left(h-\varepsilon_0\right) Q x\right)\right\rangle , x\right\rangle\\
= & (p-1)|\nabla u(F(x, t), t)|^{-p}\left\langle\nabla\left[|\nabla u(F(x, t), t)|\left(h-\varepsilon_0\right) Q\right] , x\right\rangle \\
= & (p-1)\left(h-\varepsilon_0\right) Q|\nabla u(F(x, t), t)|^{-p}\langle\nabla|\nabla u(F(x, t), t)|, x\rangle+(p-1) Q|\nabla u(F(x, t), t)|^{1-p} \langle\nabla h , x \rangle\\
& +(p-1)\left(h-\varepsilon_0\right)|\nabla u(F(x, t), t)|^{1-p} \langle\nabla Q , x \rangle\\
= &(p-1)\left(h-\varepsilon_0\right) Q|\nabla u(F(x, t), t)|^{-p-1}\left\langle\nabla u(F(x, t), t) \nabla^2 U(F(x, t), t), x \right\rangle\\
& +(p-1) Q|\nabla u(F(x, t), t)|^{1-p} \langle\nabla h , x \rangle\\
&+(p-1)\left(h-\varepsilon_0\right)|\nabla u(F(x, t), t)|^{1-p}\bigg\langle\left(\frac{-\sum_i\left(h_i\right)_t e^i}{h-\varepsilon_0}+Q x-Q \frac{\nabla h}{h-\varepsilon_0}\right), x \bigg\rangle\\
\leq & c_3 Q|\nabla u(F(x, t), t)|^{-p}\left|\nabla^2 u(F(x, t), t)\right|+c_4 Q|\nabla u(F(x, t), t)|^{1-p},
\end{align*}
hence, there exist positive constants $C_2$ and $C_3$, independent of $t$, such that
\begin{align}\label{Eq:udgj}
& \frac{\partial\mid \nabla u(F(x, t), t)\mid^{1-p}}{\partial t}
 \leq & C_2 Q|\nabla u(F(x, t), t)|^{-p}\left|\nabla^2 u(F(x, t), t)\right|+C_3Q|\nabla u(F(x, t), t)|^{1-p} .
\end{align}
Thus, by (\ref{Eq:qyjgj}), (\ref{Eq:fzhsgj}), (\ref{Eq:gsqlgj}) and (\ref{Eq:udgj}), we obtain
\begin{align*}
\frac{\partial Q}{\partial t}\leq& \frac{f(x)}{h-\varepsilon_0}\left[C_1Q(x_0,t)\frac{h(x, t)  \mathcal{K}(x, t)}{\mid \nabla u(F(x, t), t)\mid^{p-1}}-Qh\frac{\eta(t)\mathcal{K}(x, t)}{\mid \nabla u(F(x, t), t)\mid^{p-1}}\right.\\
&+\eta(t)h(x, t)\mathcal{K}(x, t)\left(C_2 Q|\nabla u(F(x, t), t)|^{-p}\left|\nabla^2 u(F(x, t), t)\right|+C_3Q|\nabla u(F(x, t), t)|^{1-p}\right) \\
&\left.+\mathcal{K} Q\left((n-1)-\varepsilon_0(n-1) \mathcal{K}^{\frac{1}{n-1}}\right)\frac{\eta(t)h(x, t)}{\mid \nabla u(F(x, t), t)\mid^{p-1}}\right]\\
& +Q+Q^2 .
\end{align*}

For a $Q$ large enough, by  Lemma \ref{lemco} and Lemma \ref{lemcqgj}, it is easy to see that there exist a positive constant $C$, independent of $t$, such that  
$$\frac{1}{C} \mathcal{K} \leq Q \leq C \mathcal{K},$$ 
then, for $\mathcal{K} \approx Q>>1$ ( $Q$ is sufficiently large), we obtain 
\begin{align}\label{Eq:cwffx}
\frac{\partial Q}{\partial t} \leq C_0 Q^2\left(C_1-\varepsilon_0 Q^{\frac{1}{n-1}}\right)<0,
\end{align}
for some positive $C_0, C_1$, independent of $t$. Therefore, the ODE (\ref{Eq:cwffx}) implies that
$$
Q\left(x_0, t\right) \leq C,
$$
for some $C>0$, independent of $t$ and $x$, depending only on $Q(x_0,t)$, $C_0$, $C_1$, $\varepsilon_0$. Hence, we conclude
$$
\mathcal{K}\leq \frac{Q(x, t)(h-\varepsilon_0)+h}{\frac{\eta(t)f(x) h(x, t)}{\mid\nabla u(F(x, t), t)\mid^{p-1}}}\leq\frac{Q(x_0, t)(h-\varepsilon_0)+h}{\frac{\eta(t)f(x) h(x, t)}{\mid\nabla u(F(x, t), t)\mid^{p-1}}}\leq C.
$$

Next, we prove $\mathcal{K}_i\geq\frac{1}{C}$. Consider the auxiliary function
\begin{align}\label{Eq:fzhsh}
H(x, t)=\log \lambda_{\max }\left(\left\{b_{i j}\right\}\right)-A \log h(x, t)+\exp^{B\left|\nabla h\right|^2},
\end{align}
where $A$ and $B$ are positive constants, $\lambda_{\max }\left(\left\{b_{i j}\right\}\right)$ is the maximal eigenvalue of $\left\{b_{i j}\right\}$.

For any fixed $t \in(0,+\infty)$, we assume that the maximum value of $H(x, t)$ is achieved at $x_0$ on $\mathbb{S}^{n-1}$. By rotating the coordinates, we can assume that $\left\{b_{i j}\left(x_0, t\right)\right\}$ is the diagonal, $\lambda_{\max }\left(\left\{b_ {i j}\left(x_0, t\right)\right\}\right)=b_{11}\left(x_0, t\right)$ is a diagonal. To get the lower bound of the principal curvature, just show that $b_{11}$ has an upper bound. At the same time, transforming (\ref{Eq:fzhsh}) into
$$
\widetilde{H}(x, t)=\log b_{11}-A \log h(x, t)+\exp^{B\left|\nabla h\right|^2} .
$$

Thus, according to the above assumption, for any fixed $t \in(0,+\infty)$, $\widetilde{H}(x, t)$ has a local maximum at $x_0$, which states that, at $x_0$, we obtain
\begin{align*}
0=\nabla_i \widetilde{H} & =b^{11} \nabla_i b_{11}-A \frac{h_i}{h}+2 B \exp^{B\left|\nabla h\right|^2} \sum_j h_j h_{j i} \\
& =b^{11}\left(b_{1 i 1}\right)-A \frac{h_i}{h}+2 B \exp^{B\left|\nabla h\right|^2} h_i h_{i i} \\
& =b^{11}\left(h_{i 11}+h_1 \delta_{1 i}\right)-A \frac{h_i}{h}+2 B \exp^{B\left|\nabla h\right|^2} h_i h_{i i},
\end{align*}
and
\begin{align}\label{Eq:edejd}
\nonumber0 \geq \nabla_{i i} \widetilde{H}=&b^{11} \nabla_{i i} b_{11}-\left(b^{11}\right)^2\left(\nabla_i b_{11}\right)^2-A\left(\frac{h_{i i}}{h}-\frac{h_i^2}{h^2}\right) \\
&+4 B^2 \exp^{B\left|\nabla h\right|^2}\left(h_i h_{i i}\right)^2+2 B \exp^{B\left|\nabla h\right|^2}\left[\sum_j h_j h_{j i i}+h_{i i}^2\right] .
\end{align}
Since,
\begin{align}\label{Eq:egytds}
\nonumber\frac{\partial \widetilde{H}(x, t)}{\partial t}  =&b^{11} \partial_t b_{11}-A \frac{h_t}{h}+2 B \exp^{B\left|\nabla h\right|^2} \sum_j h_j h_{j t}\\
=& b^{11}\left(h_{11 t}+h_t\right) 
-A \frac{h_t}{h}+2 B \exp^{B\left|\nabla h\right|^2} \sum_j h_j h_{j t}.
\end{align}
In addition, use (\ref{eq:thgqwf}) to obtain,
\begin{align}\label{Eq:lxbliu}
\nonumber\log \left(h-h_t\right) & =\log \left(\eta(t)\frac{f(x) h(x, t)  \mathcal{K}(x, t)}{\mid \nabla u(F(x, t), t)\mid^{p-1}}\right) \\
& =-\log \operatorname{det}\left(\nabla_{\mathbb{S}^{n-1}}^2 h+h I\right)+\Psi(x, t),
\end{align}
where
\begin{align}\label{Eq:lxblf}
\Psi(x, t):=\log \left(\eta(t) \frac{f(x) h(x, t) }{|\nabla u(F(x, t), t)|^{p-1}} \right) .
\end{align}
Taking the  second  covariant derivative with respect to $e^j$ on both sides of the equation (\ref{Eq:lxbliu}), we can derive
\begin{align}\label{Eq:yjxbds}
\nonumber\frac{h_j-h_{j t}}{h-h_t} & =-\sum_{i, k} b^{i k} \nabla_j b_{i k}+\nabla_j \Psi \\
& =-\sum_i b^{i i}\left(h_{j i i}+h_i \delta_{i j}\right)+\nabla_j \Psi,
\end{align}
and
\begin{align}\label{Eq:ejxbds}
\nonumber&\frac{h_{11}-h_{11 t}}{h-h_t}-\frac{\left(h_1-h_{1 t}\right)^2}{\left(h-h_t\right)^2}\\
=&-\sum_i b^{i i} \nabla_{11} b_{i i}+\sum_{i, k} b^{i i} b^{k k}\left(\nabla_1 b_{i k}\right)^2+\nabla_{11} \Psi.
\end{align}
The Ricci identity on the sphere is written as
\begin{align}\label{Eq:qmlqh}
\nabla_{11} b_{i j}=\nabla_{i j} b_{11}-\delta_{i j} b_{11}+\delta_{11} b_{i j}-\delta_{1 i} b_{1 j}+\delta_{1 j} b_{1 i}.
\end{align}
Thus, by the  (\ref{Eq:egytds}), (\ref{Eq:ejxbds}), (\ref{Eq:205}), (\ref{Eq:qmlqh}), (\ref{Eq:edejd}) and (\ref{Eq:yjxbds}), we have at $x_0$,

\begin{align}\label{Eq:edtgj}
\nonumber\frac{\frac{\partial }{\partial t} \widetilde{H}(x, t)}{h-h_t} & =\frac{b^{11}\left(h_{11 t}+h_t\right)}{h-h_t}-A \frac{h_t}{h\left(h-h_t\right)}+2 B \exp^{B\left|\nabla h\right|^2} \frac{\sum_j h_j h_{j t}}{h-h_t} \\
\nonumber& =b^{11}\left[\frac{\left(h_{11 t}-h_{11}+h_{11}+h-h+h_t\right)}{h-h_t}\right]- \frac{A}{h} \frac{h_t-h+h}{\left(h-h_t\right)}+2 B \exp^{B\left|\nabla h\right|^2} \frac{\sum_j h_j h_{j t}}{h-h_t} \\
\nonumber& =b^{11}\left[-\frac{\left(h_1-h_{1 t}\right)^2}{\left(h-h_t\right)^2}+\sum_i b^{i i} \nabla_{11} b_{i i}-\sum_{i, k} b^{i i} b^{k k}\left(\nabla_1 b_{i k}\right)^2-\nabla_{11} \Psi\right]\\
\nonumber& +\frac{1}{h-h_t}-b^{11}+\frac{A}{h}-\frac{A}{h-h_t}+2 B \exp^{B\left|\nabla h\right|^2} \frac{\sum_j h_j h_{j t}}{h-h_t} \\
\nonumber\leq & b^{11}\left[\sum_i b^{i i}\left(\nabla_{i i} b_{11}-b_{11}+b_{i i}\right)-\sum_{i, k} b^{i i} b^{k k}\left(\nabla_1 b_{i k}\right)^2\right] \\
\nonumber& +\frac{1}{h-h_t}(1-A)-b^{11} \nabla_{11} \Psi+\frac{A}{h}+2 B \exp^{B\left|\nabla h\right|^2} \frac{\sum_j h_j h_{j t}}{h-h_t} \\
\nonumber\leq & \sum_i b^{i i}\left[\left(b^{11}\right)^2\left(\nabla_i b_{11}\right)^2+A\left(\frac{h_{i i}}{h}-\frac{h_i^2}{h^2}\right)-4 B^2 \exp^{B\left|\nabla h\right|^2}\left(h_i h_{i i}\right)^2\right. \\
\nonumber& \left.-2 B \exp^{B\left|\nabla h\right|^2}\left(\sum_j h_j h_{j i i}+h_{i i}^2\right)\right]-b^{11} \sum_{i, k} b^{i i} b^{k k}\left(\nabla_1 b_{i k}\right)^2-b^{11} \nabla_{11} \Psi+\frac{A}{h} \\
\nonumber& +2 B \exp^{B\left|\nabla h\right|^2} \frac{\sum_j h_j h_{j t}}{h-h_t}+\frac{1}{h-h_t}(1-A) \\
\nonumber\leq & \sum_i b^{i i} d\left(\frac{h_{i i}+h-h}{h}-\frac{h_i^2}{h^2}\right)-2 B \exp^{B\left|\nabla h\right|^2} \sum_i b^{i i} h_{i i}^2 \\
\nonumber& +2 B \exp^{B\left|\nabla h\right|^2} \sum_j h_j\left[-\sum_i b^{i i} h_{j i i}+\frac{h_{j t}}{h-h_t}\right] \\
\nonumber& -4 B^2 \exp^{B\left|\nabla h\right|^2} \sum_i b^{i i} h_i^2\left(b_{i i}-h\right)^2-b^{11} \nabla_{11} \Psi+\frac{A}{h}+\frac{1}{h-h_t}(1-A)\\
\nonumber& \leq-b^{11} \nabla_{11} \Psi-2 B \exp^{B\left|\nabla h\right|^2} \sum_j h_j \nabla_j \Psi+\left(2 B \exp^{B\left|\nabla h\right|^2}\left|\nabla h\right|^2-A\right) \sum_i b^{i i} \\
\nonumber&-2 B \exp^{B\left|\nabla h\right|^2} \sum_i b_{i i}-4 B^2 \exp^{B\left|\nabla h\right|^2} \sum_i b_{i i} h_i^2+8 h B^2 \exp^{B\left|\nabla h\right|^2}\left|\nabla h\right|^2 \\
&+\frac{2 B \exp^{B\left|\nabla h\right|^2}\left|\nabla h\right|^2+1-A}{h-h_t}+4(n-1) A B \exp^{B\left|\nabla h\right|^2}+\frac{n A}{h} .
\end{align}
Then, differentiate from (\ref{Eq:lxblf}) and get
\begin{align}\label{Eq:fdy1}
\nabla_j \Psi=\frac{f_j}{f}+(1-p) |\nabla u|^{-p}|\nabla u|_j+\frac{h_j}{h},
\end{align}
thus
\begin{align}\label{Eq:fdy2}
\nabla_{11} \Psi= & \frac{f f_{11}-f_1^2}{f^2}-(1-p) |\nabla u|^{-p-1}\left(|\nabla u|_1\right)^2 +(1-p)|\nabla u|^{-p}|\nabla u|_{11}+\frac{h h_{11}-h_1^2}{h^2} .
\end{align}
Since  
$$
|\nabla u(F(x, t), t)|=\langle-\nabla u(F(x, t), t), x\rangle.
$$
Taking the covariant derivative of both sides, we have
\begin{align}\label{Eq:udy1}
\nonumber|\nabla u|_j & =-\langle\nabla u , e^j\rangle-\langle\left(\nabla^2 u\right) F_j , x\rangle \\
& =-\left\langle\sum_i b_{i j}\left(\nabla^2 u\right) e^i, x\right\rangle,
\end{align}
thus
\begin{align*}
|\nabla u|_{11}= & - \left\langle\sum_i b_{i 11}\left(\nabla^2 u\right) e^i , x\right\rangle-\left\langle\sum_{i, j} b_{j 1} b_{i 1} \left(\nabla^3 u\right) e^j e^i , x\right\rangle+\left\langle\sum_i b_{i 1} \delta_{1 i}\left(\nabla^2 u\right) x, x\right\rangle \\
& -\left\langle\sum_i b_{i 1}\left(\nabla^2 u\right) e^i  e_1\right\rangle.
\end{align*}
Below, we estimate the first and second items of 
(\ref{Eq:edtgj}). By (\ref{Eq:fdy1}), Lemma \ref{lemco} and Lemma \ref{lemcqgj}, we obtain 
\begin{align*}
-  2 B \exp^{B\left|\nabla h\right|^2} \sum_j h_j \nabla_j \Psi
= & -2 B \exp^{B\left|\nabla h\right|^2} \sum_j h_j\left[\frac{f_j}{f}+(1-p) |\nabla u|^{-p}|\nabla u|_j+\frac{h_j}{h}\right] \\
\leq & C_1 B \exp^{B\left|\nabla h\right|^2}+2(1-p)B \exp^{B\left|\nabla h\right|^2}|\nabla u|^{-p}\bigg\langle\sum_{j} h_jb_{j j} \left(\nabla^2 u\right) e^j , x\bigg\rangle.
\end{align*}
Using (\ref{Eq:fdy2}),  Lemma \ref{lemco} and Lemma \ref{lemcqgj}, for  $1<p<\infty$, we have 
\begin{align*}
&-b^{11} \nabla_{11} \Psi \\
= &-b^{11}\left[\frac{f f_{11}-f_1^2}{f^2}-(1-p) |\nabla u|^{-p-1}\left(|\nabla u|_1\right)^2  +(1-p)|\nabla u|^{-p}|\nabla u|_{11}+\frac{h h_{11}-h_1^2}{h^2}\right]\\
= &-b^{11}\left[\frac{f f_{11}-f_1^2}{f^2}-(1-p) |\nabla u|^{-p-1}\left\langle-\sum_i b_{i j}\left(\left(\nabla^2 u\right) e^i , x\right)\right\rangle^2  \right.\\
&\left.+(1-p)|\nabla u|^{-p}\left(\bigg\langle-\sum_i b_{i 11} \left(\nabla^2 u\right) e^i , x\bigg\rangle-\bigg\langle\sum_{i, j} b_{j 1} b_{i 1}\left(\nabla^3 u\right) e^j e^i , x\bigg\rangle\right.\right.\\
&\left.\left.+\bigg\langle\sum_i b_{i 1} \delta_{1 i}\left(\nabla^2 u\right) x, x\bigg\rangle-\bigg\langle\sum_i b_{i 1}\left(\nabla^2 u\right) e^i , e_1\bigg\rangle\right)+\frac{h (b_{11}-h)-h_1^2}{h^2}\right]\\
= &b^{11}\left[-\frac{f f_{11}-f_1^2}{f^2}+(1-p) |\nabla u|^{-p-1}\bigg\langle\sum_i b_{i j}\left(\nabla^2 u\right) e^i , x\bigg\rangle^2  \right.\\
&\left.+\frac{(1-p)}{|\nabla u|^p}\bigg\langle\sum_i b_{i 11}\left(\nabla^2 u\right) e^i, x\bigg\rangle+\frac{(1-p)}{|\nabla u|^p}\bigg\langle\sum_{i, j} b_{j 1} b_{i 1}\left(\nabla^3 u\right) e^j e^i , x\bigg\rangle\right.\\
&\left.-\frac{(1-p)}{|\nabla u|^p}\bigg\langle\sum_i b_{i 1} \delta_{1 i}\left(\nabla^2 u\right) x , x\bigg\rangle+\frac{(1-p)}{|\nabla u|^p}\bigg\langle\sum_i b_{i 1}\left(\nabla^2 u\right) e^i , e_1\bigg\rangle+\frac{h (b_{11}-h)-h_1^2}{h^2}\right]\\
\leq&b^{11}\left[C_3 +\frac{(1-p)}{|\nabla u|^p}\bigg\langle\sum_i b_{i 11}\left(\nabla^2 u\right) e^i , x\bigg\rangle+(1-p)(n-1) b_{1 1}^2\frac{|\nabla^3 u|}{|\nabla u|^p} \right.\\
&\left.+2(1-p)(n-1)b_{11}\frac{|\nabla^2 u|}{|\nabla u|^p}-\frac{b_{11}}{h^2}+1+\frac{h_1^2}{h^2}\right].
\end{align*}
Hence, by  Lemma \ref{lemejtg}, Lemma \ref{lemco} and Lemma \ref{lemcqgj}, we have 
\begin{align}\label{Eq:lgzh}
\nonumber& -2 B \exp^{B\left|\nabla h\right|^2} \sum_j h_j \nabla_j \Psi-b^{11} \nabla_{11} \Psi \\
\nonumber\leq& C_1 B \exp^{B\left|\nabla h\right|^2}+\frac{(1-p)}{|\nabla u|^p}\bigg\langle\sum_ib^{11} b_{i 11}\left(\nabla^2 u\right) e^i , x\bigg\rangle\\
\nonumber&+b^{11}\left[C_3 +(1-p)(n-1) b_{1 1}^2\frac{|\nabla^3 u|}{|\nabla u|^p}+2(1-p)(n-1)b_{11}\frac{|\nabla^2 u|}{|\nabla u|^p}-\frac{b_{11}}{h^2}+1+\frac{h_1^2}{h^2}\right]\\
\leq& \widetilde{C}_1 B \exp^{B\left|\nabla h\right|^2}+\widetilde{C}_2 b^{11}+\widetilde{C}_3+\widetilde{C}_4 b_{11}+\widetilde{C}_5 d
\end{align}

Substituting (\ref{Eq:lgzh}) into (\ref{Eq:edtgj}), and choosing $A=2 B \max _{\mathbb{S}^{n-1} \times(0, \infty)} \exp^{B\left|\nabla h\right|^2}\left|\nabla h\right|^2+1$, with a suitable $B>\widetilde{C}_4$, then, at $x_0$, we have
\begin{align}\label{Eq:gje}
\frac{\frac{\partial }{\partial t} \widetilde{H}(x, t)}{h-h_t} \leq & \widetilde{C}_1 B \exp^{B\left|\nabla h\right|^2}+\widetilde{C}_2 b^{11}+\widetilde{C}_3+\widetilde{C}_4 b_{11}+\widetilde{C}_5+\widetilde{C}_6 B \\
\nonumber& -2 B \exp^{B\left|\nabla h\right|^2} \sum_i b_{i i}+\widetilde{C}_7 B^2 \exp^{B\left|\nabla h\right|^2}+4(n-1) A B \exp^{B\left|\nabla h\right|^2}+\frac{n A}{h}<0,
\end{align}
provided $b_{11}>>1$. (\ref{Eq:gje}) yields
$$
H\left(x_0, t\right)=\widetilde{H}\left(x_0, t\right) \leq C
$$
for some $C>0$, independent of $t$ and $x$. This implies that the principal radii is bounded. This completes the proof of Lemma \ref{thmcegj}.
\end{proof}

\section{The existence of smooth solution}\label{sec4}
In this section, we will give the proof of the main theorem.
\begin{proof}[Proof of Theorem \ref{thm11}] The uniform estimates of support function and principal curvature showed in  Lemma \ref{lemco}, Lemma \ref{lemcqgj} and Lemma \ref{thmcegj} imply that (\ref{eq:thgqwf}) is uniformly parabolic in $C^2$ norm space. Then, by virtue of the standard Krylov's regularity theory \cite{KRY1987} of uniform parabolic equation, the estimates of higher derivatives can be naturally obtained. Hence, we obtain that the long-time existence and regularity of the solution of (\ref{eq:thgqwf}). Furthermore, there exists a uniformly positive constant $C$, independent of $t$, such that
\begin{align}\label{Eq:lerg}
\|h\|_{C_{x, t}^{i, j}\left(\mathbb{S}^{n-1} \times[0,+\infty)\right)} \leq C
\end{align}
for each pair of nonnegative integers $i$ and $j$.
From Lemma \ref{lemddx}, we have
\begin{align}\label{Eq:nhkp}
\frac{d \Psi\left(\Omega_t\right)}{d t} \leq 0,
\end{align}
thus, by  (\ref{Eq:nhkp}), if there exists a $t_0$ such that
$$
\left.\frac{d \Psi\left(\Omega_t\right)}{d t}\right|_{t=t_0}=0,
$$
then,
$$
\mid \nabla u(F(x, t_0), t_0)\mid^{p-1}\frac{1}{\mathcal{K}(x, t_0)}=\frac{ \Gamma\left(\Omega_{t_0}\right)}{\int_{\mathbb{S}^{n-1}} f(x) h(x, t_0) d x}f(x).
$$
Let $\Omega=\Omega_{t_0}$, thus, $\Omega$ satisfies (\ref{eq:thmaj}).

In addition, suppose that for every $t>0$,
\begin{align}\label{Eq:psdf}
\frac{d \Psi\left(\Omega_t\right)}{d t}<0 .
\end{align}
According to (\ref{Eq:lerg}), applying the Arzel\`{a}-Ascoli theorem and diagonal argument, there exists a subsequence of $t$, denoted as $\left\{t_k\right\}_{k \in\mathbb{ N}} \subset(0,+\infty)$, and there is a smooth function $h(x)$, such that
\begin{align}\label{Eq:psty}
\left\|h\left(x, t_k\right)-h(x)\right\|_{C^i\left(\mathbb{S}^{n-1}\right)} \rightarrow 0
\end{align}
uniformly for each nonnegative integer $i$ as $t_k \rightarrow+\infty$. This shows that $h(x)$ is a support function. We use $\Omega$ to represent the convex body determined by $h(x)$. Therefore, $\Omega$ is a smooth, origin symmetric and strictly convex body.

Using (\ref{Eq:lerg}) and the consistent estimates in Lemma \ref{lemco}, Lemma \ref{lemcqgj} and Lemma \ref{thmcegj}, we get $\Psi\left(\Omega_t\right)$ is a bounded function in $t$, and $\frac{d \Psi\left(\Omega_t\right)}{d t}$  is uniformly continuous. So, for any $t>0$, using (\ref{Eq:psdf}), there exist positive constants $C$, independent of $t$, such that
\begin{align}\label{Eq:p1sty}
\int_0^t\left(-\frac{d \Psi\left(\Omega_t\right)}{d t}\right) d t=\Psi\left(\Omega_0\right)-\Psi\left(\Omega_t\right) \leq C,
\end{align}
thus, we have
\begin{align*} 
\int_0^{+\infty}\left(-\frac{d \Psi\left(\Omega_t\right)}{d t}\right) d t \leq C,
\end{align*}
this shows that there is a subsequence $t_k \rightarrow+\infty$ of  $t$, such that
\begin{align}\label{Eq:yuzj}
\left.\frac{d \Psi\left(\Omega_t\right)}{d t}\right|_{t=t_k} \rightarrow 0 \quad \text { as } t_k \rightarrow+\infty .
\end{align}
Using again the estimates established in Lemma \ref{lemco}, Lemma \ref{lemcqgj} and Lemma \ref{thmcegj}, we can obtain that there exists a positive constant $\lambda>0$, such that

\begin{align}\label{Eq:yuij}
\nonumber\left.\frac{d \Psi\left(\Omega_t\right)}{d t}\right|_{t=t_k} & =-\left.\int_{\mathbb{S}^{n-1}}\frac{\bigg(-\frac{f(x)\Gamma(\Omega_t)|\nabla u(F(x, t), t)|^{-(p-1)}\mathcal{K}h}{\int_{\mathbb{S}^{n-1}} f(x) h(x, t) d x}+h\bigg)^2}{\Gamma(\Omega_t)|\nabla u(F(x, t), t)|^{-(p-1)}\mathcal{K}h}dx\right|_{t=t_k} \\
& \leq-\left.\lambda \int_{\mathbb{S}^{n-1}}\left[-f h \mathcal{K} \frac{1}{|\nabla u(F(x, t), t)|^{p-1}} \frac{\Gamma(\Omega_t)}{\int_{\mathbb{S}^{n-1}} f h d x}+h\right]	^2 d x\right|_{t=t_k} .
\end{align}
Taking the limit $t_k \rightarrow+\infty$ in (\ref{Eq:yuij}), by (\ref{Eq:psty}) and (\ref{Eq:yuzj}), we can deduce that
$$
0=\left.\lim _{t_k \rightarrow+\infty} \frac{d \Psi\left(\Omega_t\right)}{d t}\right|_{t=t_k} \leq-\lambda \int_{\mathbb{S}^{n-1}}\left[-f h \mathcal{K} \frac{1}{|\nabla u(F(x, t), t)|^{p-1}} \frac{\Gamma(\Omega_t)}{\int_{\mathbb{S}^{n-1}} f h d x}+h\right]^2 d x \leq 0,
$$
which implies that
$$\mid \nabla u(F(x))\mid^{p-1}\frac{1}{\mathcal{K}(x)}=f(x)\frac{ \Gamma\left(\Omega\right)}{\int_{\mathbb{S}^{n-1}} f(x) h(x) d x}.$$
Therefore, $h(x)$ is the solution to equation (\ref{eq:thmaj}). This completes the proof of Theorem \ref{thm11}.
\end{proof}


\begin{thebibliography}{10}

\bibitem{AKM2023} M. Akman and S. Mukherjee, \emph{On the Minkowski problem for $p$-harmonic measures}, (2023), arXiv preprint arXiv:2310.10247.

\bibitem{BIJ1994} I. J. Bakelman, \emph{Convex Analysis and Nonlinear Geometric Elliptic Equations}, Springer, Berlin, 1994.

\bibitem{BLYZ2012} K. J. B\"{o}r\"{o}czky, E. Lutwak, D. Yang and G. Y. Zhang, \emph{The log-Brunn-Minkowski inequality}, Adv. Math., {\bf 231} (2012), 1974-1997.

\bibitem{BLYZ2013} K. J. B\"{o}r\"{o}czky, E. Lutwak, D. Yang and G. Zhang, \emph{The logarithmic Minkowski problem}, J. Amer. Math. Soc., {\bf 26} (2013), 831-852.

\bibitem{BRF2024} K. J. B\"{o}r\"{o}czky, J. P. G. Ramos and A. Figalli, The Isoperimetric inequality, the Brunn-Minkowski theory and Minkowski type Monge-Amp\`ere equations on the sphere, 2024, https://users.renyi.hu/~carlos/Brunn-Minkowski-Book-2024-02-05.pdf.

\bibitem{CLYZ2009} A. Cianchi, E. Lutwak, D. Yang, and G. Zhang,  \emph{Affine Moser-Trudinger and Morrey-Sobolev
inequalities}, Calc. Var. Partial Differential Equations, {\bf 36} (2009), 419-436.

\bibitem{CAM2010} A. Colesanti and M.Fimiani, \emph{The Minkowski problem for torsional rigidity}, Indiana Univ. Math. J.,  {\bf 59} (2010), 3: 1013-1039.

\bibitem{CHZ2019} C. Chen,  Y. Huang and Y. Zhao, \emph{Smooth solutions to the dual Minkowski problem}, Math. Ann., {\bf 373} (2019), 953-976 

\bibitem{CFL2023} S. Chen, Y. Feng and W. Liu, \emph{Uniqueness of solutions to the logarithmic Minkowski problem in $\mathbb{R}^3$}, Adv. Math., {\bf 411, Part A} (2022), 108782.

\bibitem{CLZ2019} S. Chen, Q.-R. Li and G. Zhu, \emph{The logarithmic Minkowski problem for non-symmetric measures}, Trans. Amer. Math. Soc., {\bf 371} (2019), 2623-2641.

\bibitem{GHWXY1}  R. J. Gardner, D. Hug, W. G. Weil, S. Xing and D. Ye, {\it General volumes in the Orlicz-Brunn-Minkowski theory and a related Minkowski problem I}, Calc. Var. Partial Differential Equations, {\bf 58} (2019), DOI: 10.1007/s00526-018-1449-0.

\bibitem{GHWXY2}   R. J. Gardner, D. Hug, S. Xing and D. Ye, \emph{General volumes in the Orlicz-Brunn-Minkowski theory and
a related Minkowski problem II},  Calc. Var. Partial Differential Equations, {\bf 59} (2020), DOI: 10.1007/s00526-019-1657-2.

\bibitem{KLL2023} L. Kryvonos and D. Langharst, \emph{Measure theoretic Minkowski's existence theorem and projection bodies}, Trans. Amer. Math. Soc., {\bf 376} (2023), 8447-8493.

\bibitem{HLYZ2010} C. Haberl, E. Lutwak, D. Yang and G. Zhang, \emph{The even Orlicz Minkowski problem}, Adv. Math., {\bf 224} (2010), 2485-2510.	

\bibitem{HJR2024} J. Hu,  \emph{A Gauss curvature flow approach to the torsional Minkowski problem}, J. Differential Equations,  {\bf 385}  (2024),  254-279.

\bibitem{HLYZ2016} Y. Huang, E. Lutwak, D. Yang and G. Zhang, \emph{Geometric measures in the dual Brunn-Minkowski theory and their associated Minkowski problems}, Acta Math., {\bf 216} (2016), 325-388.

\bibitem{HXZ2021} Y. Huang, D. Xi and Y. Zhao, \emph{The Minkowski problem in Gaussian probability space}, Adv. Math., {\bf 385} (2021), 107769.

\bibitem{JER1989} D. Jerison,   \emph{Harmonic measure in convex domains},  Bull. Amer. Math. Soc. (N.S.), {\bf 21} (1989), 2: 255-260.

\bibitem{JER1991} D. Jerison,   \emph{Prescribing harmonic measure on convex domains}, Invent. Math., {\bf 105} (1991), 1: 375-400.

\bibitem{JER1996} D. Jerison, \emph{A Minkowski problem for electrostatic capacity}, Acta Math.,  {\bf 176} (1996), 1: 1-47.

\bibitem{KRY1987}  N. V. Krylov, Nonlinear elliptic and parabolic equations of the second order, Dordrecht, Holland: D. Reidel Publishing Company, 1987.

\bibitem{LWG2023} C. Li and G. Wei, \emph{Existence of solution for $L_p$-Minkowski problem of $0< p< 1$ with measures in  $\mathbb{R}^n$}, Int. J. Math., {\bf 34} (2023), 3: 2350009.

\bibitem{LIV2019} G. V. Livshyts, \emph{An extension of Minkowski's theorem and its applications to questions about projections for measures}, Adv. Math., {\bf 356} (2019), 106803.

\bibitem{LNYZ2024} N. Li, D. Ye and B. Zhu, \emph{The dual Minkowski problem for unbounded closed convex sets}, Math. Ann., {\bf 388} (2024),  2001-2039.

\bibitem{LYU2019} Y. Liu and J. Lu, \emph{A flow method for the dual Orlicz-Minkowski problem}, Trans. Am. Math. Soc., {\bf 373} (2020), 8: 5833-5853.

\bibitem{LNM2010} J. Lewis and K. Nystr\"{o}m K  \emph{Boundary behavior and the Martin boundary problem for $p$ harmonic functions in Lipschitz domains}, Ann. of Math., {\bf 172} (2010), 1907-1948.

\bibitem{LUT1993} E. Lutwak, \emph{The Brunn-Minkowski-Firey theory I: Mixed volumes and the Minkowski problem}, J. Differential Geom., {\bf 38} (1993), 1: 131-150.

\bibitem{LXYZ2024} E. Lutwak, D. Xi, D. Yang  and G. Zhang, \emph{Chord measures in integral geometry and their Minkowski problems}, Comm. Pure Appl. Math., (2023), DOI: 10.1002/cpa.22190.

\bibitem{LYZ2002} E. Lutwak, D. Yang, and G. Zhang, \emph{Sharp affine $L_p$ Sobolev inequalities}, J. Differential
Geom., {\bf 62} (2002), 17-38.

\bibitem{SRA2014} R. Schneider, {\it Convex Bodies: The Brunn-Minkowski theory}, 2nd edn, Cambridge Univ. Press, Cambridge, 2014.


\bibitem{UJI1991} J. I. E. Urbas, \emph{An expansion of convex hypersurfaces}, J. Differential Geom., {\bf 33} (1991), 1: 91-125.

\bibitem{XYZZ2023} D. Xi, D. Yang, G. Zhang and Y. Zhao,\emph{The $L_p$ chord Minkowski problem}, Adv. Nonlinear Stud., {\bf 23} (2023),  20220041.

\bibitem{ZXY2018} B. Zhu, S. Xing and D. Ye, \emph{The dual Orlicz-Minkowski problem}, J. Geom. Anal., {\bf 28} (2018),  3829-3855.

\bibitem{ZGX2015} G. Zhu, \emph{The centro-affine Minkowski problem for polytopes}, J. Differential Geom.,  {\bf 101} (2015), 1: 159-174.

\bibitem{ZDX2020} D. Zou and G. Xiong,  \emph{The $L_p$ Minkowski problem for the electrostatic $\mathfrak{p}$-capacity},  J. Differential Geom., {\bf 116} (2020), 3: 555-596.

\end{thebibliography}
\end{document}